\documentclass[12pt]{article}
\usepackage{amssymb, amsmath, amsthm, amsfonts, tikz, geometry, ytableau, mathrsfs}
\usepackage[indent,margin=1cm]{caption}
\usepackage{fullpage}
\usepackage[colorlinks,linkcolor=blue,anchorcolor=blue,citecolor=blue]{hyperref}

\numberwithin{equation}{section}
\numberwithin{figure}{section}

\hypersetup{colorlinks = true}
\newtheorem{theorem}{Theorem}[section]
\newtheorem{lemma}[theorem]{Lemma}

\newtheorem{proposition}[theorem]{Proposition}
\newtheorem{problem}[theorem]{Problem}
\newtheorem{example}[theorem]{Example}

\newcommand{\Des}{{\mathrm{Des}}}
\newcommand{\des}{{\mathrm{des}}}

\begin{document}
\begin{center}
	{\large \bf Combinatorial relations among restricted and half Eulerian polynomials of types $A$, $B$, and $D$}
\end{center}
\begin{center}
Zhong-Xue Zhang\\[6pt]
\end{center}

\begin{center}
Mathematics Teaching and Research Section, Basic Department\\
Naval University of Engineering\\
Wuhan, Hubei 430030, P. R. China\\[6pt]

Email: {\tt zhzhx@mail.nankai.edu.cn}
\end{center}

\noindent\textbf{Abstract.}
In this paper, we study relations among several types of Eulerian polynomials from a combinatorial viewpoint. We establish an identity between the restricted Eulerian polynomials of types $A$ and $B$. As an application, we present a bijective proof of a new identity involving the Eulerian polynomials of type $A$ and type $B$, solving a recent open problem proposed by Zhang. Additionally, we derive an identity between the half Eulerian polynomials of type $B$ and type $D$. Using this identity, we further obtain another relation about the Eulerian polynomials of type $A$ and type $B$, as well as a recursive formula connecting the restricted Eulerian polynomials of type $D$ and Eulerian polynomials of types $A$ and $B$.  \\[6pt]
\noindent \emph{AMS Mathematics Subject Classification 2020:} 05A05, 05A19

\noindent \emph{Keywords:} Eulerian polynomials, restricted Eulerian polynomials, half Eulerian polynomials, recurrence relations, bijective proof.

\section{Introduction}\label{intro}

Eulerian polynomials are fundamental objects in algebraic combinatorics and geometry, with a wealth of generalizations and variants studied in the literature (see \cite{Brenti94, DPS09, P19, VW13, Zhu20} and references therein).
In this paper, we focus on two such variants: restricted Eulerian polynomials (conditioned on permutations ending with a fixed element) and half Eulerian polynomials (conditioned on permutations ending with a positive or negative element).
While the Eulerian polynomials of types $A$ and $B$ are connected by known identities, the combinatorial relations for restricted Eulerian polynomials and half Eulerian polynomials
have not been systematically established.  We address this gap by deriving explicit combinatorial relations among these polynomials.

Let $\mathfrak{S}_X$ denote the set of permutations of a totally ordered $n$-element set $X$.
For a permutation $\pi=\pi_1\pi_2\cdots\pi_n\in\mathfrak{S}_X$, we define the descent set $\Des(\pi) = \{\, i\in[n-1] \mid \pi_i > \pi_{i+1} \,\}$ and write $\des(\pi) = |\Des(\pi)|$.
Let $[n]=\{1,2,\ldots,n\}$. If $X=[n]$, then we write $\mathfrak{S}_n$ for $\mathfrak{S}_{[n]}$.
The Eulerian polynomial of type $A$ is
\[
A_n(t) = \sum_{\pi\in\mathfrak{S}_n} t^{\des(\pi)}.
\]
Given a totally ordered finite set $Y$, one can define the reduction of $Y$, denoted $\mathrm{red}(Y)$, as the set obtained from $Y$ by replacing its $i$-th smallest entry with $i$. In the same way one can define the reduction of any permutation $\pi\in \mathfrak{S}_Y$, denoted $\mathrm{red}(\pi)$.
Clearly, we have $\Des(\pi)=\Des(\mathrm{red}(\pi))$. It is well-known that $A_n(t)$ is real-rooted, originally established by Frobenius \cite{F1910}. Moreover, $A_{n-1}(t)$ interlaces $A_n(t)$; see \cite{Chow22, LW07, Petersen-book}.
These polynomials arise naturally in geometry: their coefficients form the $h$-vectors of Coxeter complexes of type $A$; see \cite{Petersen-book} for further details.
The geometric viewpoint naturally extends to other finite Coxeter groups. In particular, the Eulerian polynomials of type $B$ correspond to the $h$-vectors of Coxeter complexes of type $B$ \cite{Petersen-book}.

Given a finite set $X$ of positive integers, let $\mathfrak{B}_X$ denote the set of all signed permutations of $X$, and  let $\mathfrak{C}_X$ denote the set of all signed sets of $X$. The hyperoctahedral group $\mathfrak{B}_{[n]}$, denoted $\mathfrak{B}_n$, consists of all signed permutations $\pi=\pi_1\pi_2\cdots\pi_n$ of $[n]$ in window notation.
For the convenience, we sometimes use $\bar{i}$ to represent $-i$.
For $\pi\in \mathfrak{B}_n$, let $\pi_0 := 0$. The type $B$ descent set is $\Des_B(\pi) = \{\, i\in\{0,1,\ldots,n-1\}\mid \pi_i > \pi_{i+1} \,\}$ and we set $\des_B(\pi) = |\Des_B(\pi)|$.
The corresponding Eulerian polynomial of type $B$ is
\[
B_n(t) = \sum_{\pi\in \mathfrak{B}_n} t^{\des_B(\pi)},
\]
which is real-rooted \cite{Brenti94}, and $B_{n-1}(t)$ interlaces $B_n(t)$; see \cite{Chow22, Guo24}.
For a set $Y$ of integers with distinct absolute values, we define the reduction of
$Y$, denoted $\mathrm{red}(Y)$, to be the set obtained from $Y$ by replacing the entry with the $i$-th smallest absolute value with $i$ or $-i$ according to its sign.
In the same way, we can define the reduction of a permutation $\pi\in \mathfrak{S}_Y$ and denote it as $\mathrm{red}(\pi)$.
For example, if $Y=\{\bar{2},4,\bar{5},9\}$ and $\pi=(4,\bar{5},\bar{2},9)$, then $\mathrm{red}(Y)=\{\bar{1},2,\bar{3},4\}$ and $\mathrm{red}(\pi)=(2,\bar{3},\bar{1},4)$. For any permutation $\pi$ of $Y$, let $\Des_B(\pi)=\Des_B(\mathrm{red}(\pi))$ and $\des_B(\pi) = \des_B(\mathrm{red}(\pi))$.
An important identity between the $A_n(t)$ and $B_n(t)$ can be found in \cite{Petersen-book} as follows:
\begin{align}\label{eq-ab-for-red}
2B_n(t^2) = (1+t)^{n+1}A_n(t)+(1-t)^{n+1}A_n(-t).
\end{align}
Santocanale provided a bijective proof for this identity \cite{S23}.

Let $\mathfrak{D}_n \subset \mathfrak{B}_n$ denote the subgroup consisting of signed permutations with an even number of negative entries. For $\pi\in \mathfrak{D}_n$, the type $D$ descent set is
\[
\Des_D(\pi) =\{0 \mid  \pi_1 + \pi_2 < 0\} \cup \{\, i\in\{1,2,\ldots,n-1\} \mid \pi_i > \pi_{i+1} \,\},
\]
with $\des_D(\pi) = |\Des_D(\pi)|$. The Eulerian polynomial of type $D$ is
\[
D_n(t) = \sum_{\pi\in \mathfrak{D}_n} t^{\des_D(\pi)}.
\]
The real-rootedness of $D_n(t)$ was conjectured by Brenti \cite{Brenti94} and first proved by Savage and Visontai \cite{SV15}.
Similarly, from a geometric perspective, the Eulerian polynomials $D_n(t)$ correspond to the $h$-vectors of the Coxeter complexes of type $D$; see \cite{Petersen-book}.
The Eulerian polynomial of type $D$ admits a fundamental decomposition in terms of Eulerian polynomials of type $A$ and type $B$ \cite{Stembridge94}:
\begin{align}\label{rela-ABD}
D_n(t) = B_n(t) - n2^{\,n-1} t A_{n-1}(t).
\end{align}
Santocanale \cite{S23} also gave a bijective proof of this identity.

Recently, Zhang \cite{Zhang25} gave the following new relation between Eulerian polynomials of type $A$ and those of type $B$.
\begin{theorem}[{\cite[Theorem 2.7]{Zhang25}}]\label{thm-Zhang}
	For positive integer $n$, we have
	\begin{align}\label{equ-ABZhang}
		B_n(t)=\sum_{m_1+2m_2+\cdots+nm_n=n}\frac{n!(t+1)^{m_1}}{\prod_{j=1}^nj^{m_j}\cdot m_j!}\cdot \prod_{i=1}^{n-1}\left(\frac{tA_i(t)2^{i+1}}{i!}\right)^{m_{i+1}}.
	\end{align}
\end{theorem}
Zhang \cite[Problem 2.8]{Zhang25} asked for a bijective proof of Theorem \ref{thm-Zhang}.
In this paper, we answer Zhang's question by considering the relation of the corresponding restricted Eulerian polynomials.

Let us contintue to recall the definitions of the restricted Eulerian polynomials of types $A$, $B$, and $D$.
Let  $\mathfrak{S}_{n,j}$ be the subset of $\mathfrak{S}_n$ consisting of all elements ending with $j$. 
The restricted Eulerian polynomial of type $A$ is
$$\mathbf{A}_{n,j}(t)=\sum_{\pi\in \mathfrak{S}_{n,j}}\pi^{\des(\pi)}.$$
Dey \cite{D23} proved that $\mathbf{A}_{n,j}(t)$ is real-rooted and that $\mathbf{A}_{n-1,j}(t)$ interlaces $\mathbf{A}_{n,j}(t)$. Conger \cite{C2010} showed that $\mathbf{A}_{n,j}(t)$ matches the restricted Eulerian polynomials studied in \cite{NPT11}, where the latter were defined from a geometric perspective. Moreover, $\mathbf{A}_{n,j}(t)$ also coincides with a refinement of $\mathbf{s}$-Eulerian polynomials $E_{n,n-j}^{(\mathbf{s})}(t)$ for $\mathbf{s}=(1,2,\ldots,n)$ \cite{SV15}.
 For more information about $\mathbf{A}_{n,j}(t)$, see \cite{BW08, ERS98, KN09, WXX10}. 

Similarly, for the hyperoctahedral group $\mathfrak{B}_n$, let $\mathfrak{B}_{n,j} = \{\,\pi_1\cdots\pi_n \in \mathfrak{B}_n \mid \pi_n = j \,\}$, where $j \in \{\pm 1, \pm 2, \ldots, \pm n\}$. The restricted Eulerian polynomial of type $B$ is
\[
\mathbf{B}_{n,j}(t) = \sum_{\pi \in \mathfrak{B}_{n,j}} t^{\des_B(\pi)}.
\]
Dey, Shankar and Sivasubramanian \cite{DSS24} obtained a recurrence for $\mathbf{B}_{n,j}(t)$ and proved its real-rootedness.
As far as we know, no relation between $\mathbf{A}_{n,j}(t)$ and $\mathbf{B}_{n,j}(t)$ has been established.
In this paper, we obtain the following result.
\begin{theorem}\label{thm-AB}
	For a positive integer $n$ and $1\leq i\leq n$, we have
	\begin{equation}\label{equ-AB}
		\mathbf{B}_{n,i}(t)=2^{n-i}\sum_{j=0}^{i-1}\binom{i-1}{j} \mathbf{A}_{n,i-j}(t).
	\end{equation}
\end{theorem}

For $j \in \{\pm 1, \pm 2, \ldots, \pm n\}$, let $\mathfrak{D}_{n,j} = \{\, \pi_1\cdots\pi_n \in \mathfrak{D}_n \mid \pi_n = j \,\}$. The restricted Eulerian polynomial of type $D$ is defined by
\[
\mathbf{D}_{n,j}(t) = \sum_{\pi \in \mathfrak{D}_{n,j}} t^{\des_D(\pi)}.
\]

We also consider half Eulerian polynomials of type $B$ and type $D$. 
Let $\mathfrak{B}_n^+=\{\pi_1\cdots\pi_n\in\mathfrak{B}_n\mid \pi_n>0\}$ and $\mathfrak{B}_n^-=\{\pi_1\cdots\pi_n\in \mathfrak{B}_n\mid \pi_n<0\}$. The
half Eulerian polynomials of type $B$ are
$$B_n^+(t)=\sum_{\pi\in\mathfrak{B}_n^+}\pi^{\des_B(\pi)},\,\,\,\,\,\,B_n^-(t)=\sum_{\pi\in\mathfrak{B}_n^-}\pi^{\des_B(\pi)}.$$
Analogously, letting $\mathfrak{D}_n^+=\{\pi_1\cdots\pi_n\in\mathfrak{D}_n\mid \pi_n>0\}$ and $\mathfrak{D}_n^-=\{\pi_1\cdots\pi_n\in \mathfrak{D}_n\mid \pi_n<0\}$, the half Eulerian polynomials of type $D$ are
$$D_n^+(t)=\sum_{\pi\in\mathfrak{D}_n^+}\pi^{\des_D(\pi)},\,\,\,\,\,\,D_n^-(t)=\sum_{\pi\in\mathfrak{D}_n^-}\pi^{\des_D(\pi)}.$$
In \cite{H16}, Hyatt established recurrences for the half Eulerian polynomials of type $B$ and type $D$, and proved their real-rootedness. An alternative proof of this property was later given in \cite{YZ}. Abram and Bastidas \cite{AB25} showed that $B_n^+(t)$ coincide with the Ehrhart $h^*$-polynomials of certain fundamental parallelepipeds. Based on this correspondence, Abram and Bastidas provided another proof of the real-rootedness of $B_n^+(t)$. In \cite{MMYY22}, Ma, Ma, Yeh and Yeh investigated the connection between half Eulerian polynomials of type $B$ and half alternating run polynomials. In the present work, we prove the following theorem.
\begin{theorem}\label{thm-BDhalf}
	For any positive integer $n$, we have
	\begin{align}\label{eq-BDhalf}
		B_n^+(t)=D_n^+(t)+nB_{n-1}^-(t),\qquad B_n^-(t)=D_n^-(t)+ntB_{n-1}^+(t).
	\end{align}
\end{theorem}

The paper is organized as follows.
In Section~\ref{AB}, we present our bijective proof of Theorem \ref{thm-Zhang}, as well as that of Theorem \ref{thm-AB}. Section~\ref{halfBD} includes a proof of  Theorem \ref{thm-BDhalf}, together with a new relation between Eulerian polynomials of type~$A$ and type~$B$.

\section{Bijective proofs of Theorem \ref{thm-Zhang} and Theorem \ref{thm-AB}}\label{AB}

In this section we aim to give a bijective proof of Theorem \ref{thm-Zhang}, and hence to answer Zhang's question \cite[Problem 2.8]{Zhang25}.
To this end, we first prove Theorem \ref{thm-AB} by using the bijective principal.

Let us begin with a combinatorial proof of the following result.

\begin{lemma}\label{lem-1AB}
	Fixing $1\leq i\leq n$, let $X_i$ be an $n$-element set satisfying:
    \begin{itemize}
		\item (i) $\{1,2,\ldots,i\}\subseteq X_i \subseteq \{1,2,\ldots,i,\pm (i+1),\ldots,\pm n\}$;
		\item (ii) for each $r\in\{i+1,i+2,\ldots,n\}$, exactly one of $r$ and $-r$ belongs to $X_i$.
	\end{itemize}
If we let
\begin{align}\label{eq-rni}
\mathfrak{R}_{n,X_i}=\{\pi_1\cdots\pi_n\in \mathfrak{S}_{X_i}\mid  \pi_n=i\},
\end{align}
then
\begin{align*}
		\sum_{\pi\in\mathfrak{R}_{n,X_i}}t^{\des_B(\pi)}=\mathbf{A}_{n,i}(t).
	\end{align*}
\end{lemma}
\begin{proof}
	To complete the proof, it suffices to construct a descent preserving bijection from $\mathfrak{R}_{n,X_i}$ to $\mathfrak{S}_{n,i}$.
	Let $X_i=\{a_1,a_2,\ldots, a_i, \ldots,a_n\}$,
	where $a_1=1,a_2=2,\ldots,a_{i}=i<a_{i+1}< a_{i+2}<\cdots <a_{i+l}$
	and ${a}_{i+l+1}< {a}_{i+l+2}<\cdots <{a}_{n}<0$.
	Define a bijection $\phi:\, X_i\to [n]$ by setting $\phi(a_r)=	r$ for any $1\leq r\leq n$.
	We then define $\Phi_{X_i} : \mathfrak{R}_{n,X_i}\to \mathfrak{S}_{n,i}$ by
	\begin{align*}
		\Phi_{X_i}(\pi_1\pi_2\cdots \pi_{n-1}\pi_n)=\phi(\pi_1)\phi(\pi_2)\cdots\phi(\pi_n).
	\end{align*}
	From this construction, it is straightforward to verify that $\Phi_{X_i}(\pi) \in \mathfrak{S}_{n,i}$ for any $\pi\in \mathfrak{R}_{n,X_i}$.
    It is also clear that $\Phi_{X_i}$ is bijective.
	It remains to prove that $\des(\Phi_{X_i}(\pi))=\des_B(\pi)$.

	Let $W=\{a_1,a_2,\ldots,a_{i+l}\}$ and $W'=\{{a}_{i+l+1},{a}_{i+l+2},\ldots,{a}_n\}$. Note that for any $a,b\in W$ we have $a>b$ if and only if $\phi(a)>\phi(b)$. The same property holds for elements of $W'$. For any $\pi\in \mathfrak{S}_{X_i}$, we use $\mathbf{W}_j(\pi)$ (abbreviated as $\mathbf{W}_j$ if the underlying permutation is clear) to denote some maximal subsequence of consecutive elements contained in $W$. Similarly, we use $\mathbf{W}'_j(\pi)$ (abbreviated as $\mathbf{W}'_j$) to denote some maximal subsequence of consecutive elements contained in $W'$. In this way $\pi$
    can be expressed as a word in the alphabet $\{\mathbf{W}_1,\mathbf{W}_2,\ldots,\mathbf{W}'_1,\mathbf{W}'_2,\ldots\}$.
    Write $\Phi(\mathbf{W}_j)$ (respectively, $\Phi(\mathbf{W}'_j)$) for the image of $\Phi_{X_i}(\pi)$ corresponding to the subsequence $\mathbf{W}_j$ (respectively, $\mathbf{W}'_j)$).
    To compare $\des(\Phi_{X_i}(\pi))$ and $\des_B(\pi)$, we are forced to consider the first position of $\pi$ and those positions where some element of $W$ immediately follows some element of $W'$ or vice versa. Since every element $\pi$ of $\mathfrak{R}_{n,X_i}$ ends with $i\in W$, there are two cases to consider:
	\begin{itemize}
		\item Case 1. $\pi$ starts with an element in $W$ and looks like
		$$\mathbf{W}_1\mathbf{W}'_1\cdots \mathbf{W}_{r-1} \mathbf{W}'_{r-1}\mathbf{W}_{r}$$
        for some $r\geq 1$.
        Recall that for any $w \in W$, $w' \in W'$, we have $w >0> w'$ and $0<\phi(w) < \phi(w')$. Consequently,
        \begin{align*}
            \des_B(\pi)=\sum_{i=1}^{r}{\des}(\mathbf{W}_i)+\sum_{i=1}^{r-1}{\des}(\mathbf{W}'_i)+(r-1),
        \end{align*}
        where the last summand $r-1$ is contributed by those $r-1$ pairs $(W_i,W_i')$ ($1\leq i\leq r-1$) since the last element of $W_i$ is always greater than the first element of $W_i'$. It is also clear that
        \begin{align*}
             \des(\Phi_{X_i}(\pi))=\sum_{i=1}^{r}{\des}(\Phi(\mathbf{W}_i))+\sum_{i=1}^{r-1}{\des}(\Phi(\mathbf{W}'_i))+(r-1),
        \end{align*}
        where  the last summand $r-1$ is contributed by those $r-1$ pairs $(\Phi(W_i'),\Phi(W_{i+1}))$ ($1\leq i\leq r-1$) since the last element of $\Phi(W_i')$ is always greater than the first element of $\Phi(W_{i+1})$.

		\item Case 2. $\pi$ starts with an element in $W'$ and looks like
		$$\mathbf{W}'_1\mathbf{W}_1\mathbf{W}'_2\mathbf{W}_2\cdots \mathbf{W}'_r\mathbf{W}_{r}$$
        for some $r\geq 1$.
        Since $\pi_1$ is negative, we have
        \begin{align*}
            \des_B(\pi)=1+\sum_{i=1}^{r}{\des}(\mathbf{W}_i)+\sum_{i=1}^{r}{\des}(\mathbf{W}'_i)+r-1.
        \end{align*}
        Similar to Case 1, one can see that
        \begin{align*}
             \des(\Phi_{X_i}(\pi))=\sum_{i=1}^{r}{\des}(\Phi(\mathbf{W}_i))+\sum_{i=1}^{r}{\des}(\Phi(\mathbf{W}'_i))+r.
        \end{align*}
	\end{itemize}

       Note that ${\des}(\mathbf{W}_i)={\des}(\Phi(\mathbf{W}_i))$ and ${\des}(\mathbf{W}'_i)={\des}(\Phi(\mathbf{W}'_i))$ for each $i\geq 1$. This implies that  $\des(\Phi_{X_i}(\pi)) = \des_B(\pi)$ holds for both cases. Since $\Phi_{X_i}$ is bijective, we get
\begin{align*}
		\sum_{\pi\in\mathfrak{R}_{n,X_i}}t^{\des_B(\pi)}
		=	\sum_{\sigma\in\mathfrak{S}_{n,i}}t^{\des(\sigma)}=\mathbf{A}_{n,i}(t).
	\end{align*}
This completes the proof.
\end{proof}

We would like to point out that $X_i$ in the above lemma is not unique though $i$ is fixed. In fact, the second condition provides some flexibility to choose the set  $X_i$.
We provide the following example to illustrate the proof.

\begin{example}
	Taking $n=8$ and $i=3$, consider $X_3=\{1,2,3,5,8,\bar{7},\bar{6},\bar{4}\}$. The map $\phi$ from Lemma \ref{lem-1AB} is
	\begin{align*}
		\phi(1)=1,\,\phi(2)=2,\,\phi(3)=3,\,\phi(5)=4,\,\phi(8)=5,\,\phi(\bar{7})=6,\,\phi(\bar{6})=7,\,\phi(\bar{4})=8.
	\end{align*}
	For $\pi=\bar{4}21\bar{7}\bar{6}583\in Y$, we have $\Phi_{X_3}(\pi)=82167453$ and $\des_B(\pi)=\des(\Phi_{X_3}(\pi))= 4$.
\end{example}

Based on Lemma \ref{lem-1AB}, we are able to give a more general result. 
The proof of this generalization is also bijective.

\begin{lemma}\label{lem-2AB}
	Fixing two integers $i$ and $j$ with $1\leq i\leq n$ any $0\leq j<i$, let $X_i^j$ be a subset of $\{\pm 1,\pm 2,\ldots,\pm (i-1),i,\pm (i+1),\ldots,\pm n\}$ satisfying the following two conditions:
	\begin{itemize}
		\item (i) for any $1\leq r\leq n$, exactly one of $r$ and $-r$ belongs to $X_i^j$;
		\item (ii) exactly $j$ negative numbers from the set $\{\pm 1,\pm 2,\ldots,\pm (i-1)\}$ belong to $X_i^j$.
	\end{itemize}
If we let
\begin{align}\label{eq-tnij}
\mathfrak{T}_{n,X_i^j}=\{\pi_1\cdots\pi_n\in \mathfrak{S}_{X_{i}^j}\mid  \pi_n=i\},
\end{align}
then
	\begin{align*}
		\sum_{\pi\in\mathfrak{T}_{n,X_i^j}}t^{\des_B(\pi)}=\mathbf{A}_{n,i-j}(t).
	\end{align*}
\end{lemma}
\begin{proof}
	Let $$X_i^j=\{\overline{a}_1,\ldots,\overline{a}_j,a_{j+1},\ldots,a_{i-1},a_i,\ldots,a_{n}\},$$
	where $-(i-1)\leq\overline{a}_1<\cdots<\overline{a}_j<0<a_{j+1}<\cdots<a_{i-1}<a_i=i$ and $a_{i+1},\ldots,a_{n}\in \{\pm (i+1),\ldots,\pm n\}$.
	Take $X_{i-j}=\{1,\ldots, i-j, -i,-i+1,\ldots,-i+j-1, a_{i+1},\ldots, a_{n}\}$.
	We first define a map $\phi: X_i^j \to X_{i-j}$ by
	\begin{align*}
		\phi(x)=\begin{cases}
			r-i-1, & x = \overline{a}_r, \ 1 \le r \le j,\\
			r-j, & x = a_r, \ j+1 \le r \le i,\\
			a_r, & x = a_r, \ i+1 \le r \le n.
			\end{cases}			
	\end{align*}
	From this construction, we see that
    $$-i=\phi(\overline{a}_1)<\cdots<\phi(\overline{a}_j)<0<\phi(a_{j+1})<\cdots<\phi(a_{i-1})<\phi(a_i)=i-j.$$
	
	To finish the proof, we next construct a bijection $\Phi$ from $\mathfrak{T}_{n,X_i^j}$ to $\mathfrak{R}_{n,X_{i-j}}$ by setting $\Phi(\pi_1\cdots\pi_n)=\phi(\pi_1)\cdots \phi(\pi_n)$, where $\mathfrak{R}_{n,X_{i-j}}$ is defined by \eqref{eq-rni}.
	For example, if we take $\pi=\overline{a}_1\cdots\overline{a}_ja_{j+1}\cdots a_{i-1}a_{i+1}\cdots a_{n} i$, then $\Phi(\pi)=(-i)\cdots  -(i-j+1) 1\cdots (i-j-1) a_{i+1}\cdots a_{n} (i-j)$. Note that $\des_B(\pi)=\des_B(\Phi(\pi))$.
	A little thought shows that $\Phi$ is a descent preserving bijection. Note that $X_{i-j}$ satisfies the conditions of
    Lemma \ref{lem-1AB}. Thus, we have
	\begin{align}\label{eq-BB}
		\sum_{\pi\in\mathfrak{T}_{n,X_i^j}}t^{\des_B(\pi)}=\sum_{\Phi(\pi)\in\mathfrak{R}_{n,X_{i-j}}}t^{\des_B(\Phi(\pi))}=\sum_{\sigma\in\mathfrak{R}_{n,X_{i-j}}}t^{\des_B(\sigma)}=\mathbf{A}_{n,i-j}(t),
	\end{align}
    as desired.
\end{proof}

We now proceed to prove Theorem~\ref{thm-AB}.

\begin{proof}[Proof of Theorem \ref{thm-AB}]
Let $\mathfrak{O}_{n,i,j}$ denote the set of all possible $n$-element $X_{i}^j$ satisfying the conditions of Lemma~\ref{lem-2AB}.
By definition we find that
	\begin{align}
		\mathbf{B}_{n,i}(t)=\sum_{j=0}^{i-1} \sum_{X_{i}^j\in \mathfrak{O}_{n,i,j}} \sum_{\pi\in\mathfrak{T}_{n,X_{i}^j}}t^{\des_B(w)}.
	\end{align}
Note that for fixed $i$ and $j$ there are $2^{n-i} \binom{i-1}{j}$ ways to choose $X_i^j$, since we are free to choose
$k$ or $\bar{k}$ for each $k\in\{i+1,\ldots,n\}$ and we are also free to choose any $j$ numbers from $\{-1,-2,\ldots,-i+1\}$.
Furthermore, Lemma~\ref{lem-2AB} tells us that for different choices of $X_i^j$ the generating polynomials
$$\sum_{\pi\in\mathfrak{T}_{n,X_{i}^j}}t^{\des_B(w)}$$
have the same value $\mathbf{A}_{n,i-j}(t)$. Thus
    \begin{align}\label{eq-Bni}
		\mathbf{B}_{n,i}(t)=2^{\,n-i}\sum_{j=0}^{i-1} \binom{i-1}{j}\mathbf{A}_{n,i-j}(t),
	\end{align}
	as desired. This completes the proof.
\end{proof}

The reader may note that the index $i$ appearing in Theorem \ref{thm-AB} is a positive integer.
In fact, due to the symmetry of signed permutations, we can get the following lemma, for which we provide a proof for completeness.

\begin{lemma}\label{lem-B}
	For any $1\leq i\leq n$, we have
	\begin{align*}
		\mathbf{B}_{n,-i}(t)=t^n\mathbf{B}_{n,i}(t^{-1}).
	\end{align*}
\end{lemma}
\begin{proof}
	Define a map $\psi:\mathfrak{B}_{n,i}\rightarrow \mathfrak{B}_{n,-i}$ by
	\begin{align*}
		\psi(\pi_1\pi_2\cdots \pi_{n-1}\pi_n)=(-\pi_1)(-\pi_2)\cdots (-\pi_{n-1})(-\pi_n).
	\end{align*}
	It is straightforward to check that $\des_B(\pi)+\des_B(\psi(\pi))=n$. This implies that
	\begin{align*}
		\mathbf{B}_{n,-i}(t)=\sum_{\pi\in \mathfrak{B}_{n,-i}}t^{\des_B(\pi)}=\sum_{\pi\in \mathfrak{B}_{n,i}}t^{n-\des_B(\pi)}=t^n\mathbf{B}_{n,i}(t^{-1}).
	\end{align*}
The proof is complete.
\end{proof}
Based on Lemma \ref{lem-B} and Theorem \ref{thm-AB}, 
we can also express $\mathbf{B}_{n,i}(t)$ in terms of the refined Eulerian polynomials $\mathbf{A}_{n,r}(t)$ for any negative integer $i$.

For our purpose we also need the following lemma.
\begin{lemma}\label{lem-general-1}
For any positive integer  $k\geq 1$, let $X=\{a_1,a_2,\ldots,a_k\}$ be a $k$-element set of integers with
$|a_1|<|a_2|<\ldots<|a_k|$. If we  let $\mathfrak{R}_{X}$ be the set of permutations of $X$ ending with $a_1$, then
\begin{align}\label{eq-gf-rx}
\sum_{\pi\in\mathfrak{R}_{X}}t^{\des_B(\pi)}=\left\{
\begin{array}{ll}
1, & \mbox{ if $k=1$ and $a_1>0$}\\
t, & \mbox{ if $k=1$ and $a_1<0$}\\
tA_{k-1}(t), & \mbox{ if $k\geq 2$}.
\end{array}
\right.
	\end{align}
\end{lemma}
\begin{proof}
The case of $k=1$ is trivial. We may assume that $k\geq 2$.
Let $\bar{X}=\{\overline{a_1},a_2,\ldots,a_k\}$. We claim that
$$\sum_{\pi\in\mathfrak{R}_{X}}t^{\des_B(\pi)}=\sum_{\pi\in\mathfrak{R}_{\bar{X}}}t^{\des_B(\pi)}.$$ To prove this claim, consider the canonical bijection $\eta$ from $\mathfrak{R}_{X}$ to $\mathfrak{R}_{\bar{X}}$ defined by $\eta(\pi_1\cdots\pi_{k-1}a_1)=\pi_1\cdots\pi_{k-1}\overline{a_1}$.
Since $k\geq 2$ and $|a_1|<|a_2|<\ldots<|a_k|$, we have $\pi_{k-1}>a_1$ iff $\pi_{k-1}>\overline{a_1}$. Thus, $\des_B(\pi)=\des_B(\eta(\pi))$ for any $\pi\in \mathfrak{R}_{X}$, and hence the claim is valid.
If $a_1>0$, then we find that
$\{\mathrm{red}(\pi) | \pi\in \mathfrak{R}_{X}\}=\mathfrak{R}_{\mathrm{red}(X)}$.
By definition the reduction set $\mathrm{red}(X)$ coincides with some set $X_1$ satisfying the conditions of Lemma \ref{lem-1AB}. In this case,
we have
\begin{align*}
\sum_{\pi\in\mathfrak{R}_{X}}t^{\des_B(\pi)}=\sum_{\pi\in\mathfrak{R}_{\mathrm{red}(X)}}t^{\des_B(\pi)}=\sum_{\pi\in\mathfrak{R}_{k,X_1}}t^{\des_B(\pi)}
=
\mathbf{A}_{k,1}(t)=tA_{k-1}(t).
\end{align*}
\end{proof}


We can now  present our bijective proof of Theorem \ref{thm-Zhang}. 


\begin{proof}[Proof of Theorem \ref{thm-Zhang}]
Let  us first rewrite \eqref{equ-ABZhang} as
	\begin{align}\label{equ-ABZhang-r}
		B_n(t)=\sum_{m_1+2m_2+\cdots+nm_n=n}
        \frac{n!}{\prod_{j=1}^n j!^{m_j}\cdot m_j!}\cdot(t+1)^{m_1} \prod_{i=1}^{n-1}\left({tA_i(t)2^{i+1}}\right)^{m_{i+1}}.
	\end{align}
In order to give a bijective proof, it is desirable to give a combinatorial interpretation of the right hand side of \eqref{equ-ABZhang-r}.
Let $\Pi_n$ be the set of all partitions of the set $[n]$.
We say that a set partition of $[n]$ is of type $\mathbf{m}=(m_1,\ldots,m_n)$ if it has $m_i$ blocks of size $i$, for $1\leq i\leq n$. Let $\Pi_n(\mathbf{m})$ be the set of all partitions $\{T_1,\ldots,T_r\}$ of $[n]$ of type $\mathbf{m}$, where $r=m_1+\cdots+m_n$.
Let $\mathfrak{F}_n$ denote the set of signed set partitions of $[n]$ with each block $S_k$ being a sequence ending with its entry which has the smallest absolute value. For example,
$\{{2},\{1,3\}\}$ is a set partition in $\Pi_3$,  $\{(-2),(-3,1)\}$ is a signed set partition in $\mathfrak{F}_3$, but $\{(2),(1,-3)\}\not\in \mathfrak{F}_3$.

The well known equality
$|\Pi_n(\mathbf{m})|=\frac{n!}{\prod_{j=1}^n j!^{m_j}\cdot m_j!}$ and Lemma \ref{lem-general-1} indicate that it is plausible to weight each block $T$ of size $i$ by
\begin{align}\label{eq-wt-t}
\mathrm{wt}(T)=\sum_{X\in \mathfrak{C}_T}\sum_{\pi\in\mathfrak{R}_{X}}t^{\des_B(\pi)}.
\end{align}
In fact, if $i=1$, say $T=\{a\}$, then $\mathfrak{C}_T=\{a,-a\}$ and $\sum_{X\in \mathfrak{C}_T}\sum_{\pi\in\mathfrak{R}_{X}}t^{\des_B(\pi)}=(1+t)$ by Lemma $\ref{lem-general-1}$. If $i\geq 2$, then
$|\mathfrak{C}_T|=2^{i}$ and for each $X\in \mathfrak{C}_T$
we have $\sum_{\pi\in\mathfrak{R}_{X}}t^{\des_B(\pi)}=tA_{i-1}(t)$. Therefore, the right hand side of \eqref{equ-ABZhang-r}
can be interpreted as
$$
\sum_{\{T_1,\ldots,T_r\}\in \Pi_n}\mathrm{wt}(T_1)\cdots \mathrm{wt}(T_r).
$$
In view of the weight function $\mathrm{wt}$ defined by \eqref{eq-wt-t}, this can be further written as
$$\sum_{\{S_1,\ldots,S_r\}\in \mathfrak{F}_n} t^{\des_B(S_1)}\cdots t^{\des_B(S_r)}.$$

Now it remains to give a bijective proof of the following identity
$$\sum_{\pi\in\mathfrak{B}_n}t^{\des_B(\pi)}=\sum_{\{S_1,\ldots,S_r\}\in \mathfrak{F}_n} t^{\des_B(S_1)+\cdots+\des_B(S_r)}.$$

Given a signed permutation $\pi = \pi_1 \pi_2 \cdots \pi_n \in \mathfrak{B}_n$, insert a right parenthesis in $\pi$ after every right-to-left absolute minimum; that is,
an element $\pi_i$ such that $|\pi_i|<|\pi_j|$ for every $j>i$.
Then insert a left parenthesis where
appropriate; that is, before every internal right parenthesis and at the beginning.
In this way $\pi$ is decomposed into an ordered sequence of nonempty blocks, say
\begin{align}\label{eq-decomp}
\pi= \pi^{(1)} \mid \pi^{(2)}\mid \cdots \mid \pi^{(r)}.
\end{align}
By our construction, $\{\pi^{(1)},\pi^{(2)},\cdots, \pi^{(r)}\}$ is a legal member of $\mathfrak{F}_n$. Define the map $\phi: \mathfrak{B}_n \rightarrow \mathfrak{F}_n$ by setting $\phi(\pi)=\{\pi^{(1)},\pi^{(2)},\ldots, \pi^{(r)}\}$.
Note that there will be a descent between the last entry of $\pi^{(i)}$ and the first entry of $\pi^{(i+1)}$ for $1\leq i\leq r-1$ if and only if the latter is negative.
Now it is routine to verify that
\[
\des_B(\pi)= \des_B(\pi^{(1)}) + \des_B(\pi^{(2)})+ \cdots + \des_B(\pi^{(r)}).
\]
For example, with $\pi=\bar{3}4\bar{1}59\bar{7}2\bar{8}6\in\mathfrak{B}_9$ we have
    $\phi(\pi)=\{\bar{3}4\bar{1},59\bar{7}2,\bar{8}6\}$. By definition one can verify that
$$\des_B(\bar{3}4\bar{1}59\bar{7}2\bar{8}6)=4,\, \des_B(\bar{3}4\bar{1})=2,\,\des_B(59\bar{7}2)=1,\,\des_B(\bar{8}6)=1.$$
To show that $\phi$ is bijective, we construct a map $\psi:\mathfrak{F}_n \rightarrow \mathfrak{B}_n$
in the following way: if $\{S_1,\ldots,S_r\}\in \mathfrak{F}_n$, we write these block sequences in increasing order of their smallest abolute values and then let
$\psi(\{S_1,\ldots,S_r\})$ denote the concatenation
of these ordered sequences. One can check that both $\phi\circ\psi$ and $\psi\circ \phi$ are identity maps. This completes the proof.

\end{proof}



\section{Proof of Theorem \ref{thm-BDhalf}}\label{halfBD}

In this section, we aim to prove Theorem \ref{thm-BDhalf}.  As a corollary, we derive a new relation between Eulerian polynomials of type $A$ and type $B$. Furthermore, we establish a corresponding relation for the restricted Eulerian polynomials of type $D$, and propose one open problem for future investigation.

In \cite{S23}, Santocanale introduced the notion of smooth permutations:	a signed permutation $\pi=\pi_1\pi_2\cdots \pi_n$ is smooth if $\pi_1\cdot \pi_2>0$; and non-smooth otherwise. Santocanale proved that the smooth signed permutations having $k$ type~$B$ descents are in bijection with even signed permutations having $k$ type $D$ descents. Motivated by this result, we obtain the following refinement.
\begin{lemma}\label{lem-half-smooth}
	For any $0\leq k\leq n$, the smooth signed permutations in $\mathfrak{B}_n^{+}$ with $k$ type $B$ descents are in bijection with the permutations in $\mathfrak{D}_n^{+}$ with $k$ type~$D$ descents.
\end{lemma}
\begin{proof}
	Let $X=\{\pi\in\mathfrak{B}_n^{+}\mid \des_B(\pi)=k\mbox{ and $\pi$ is smooth }\}$ and $Y=\{\pi\in\mathfrak{D}_n^{+}\mid \des_D(\pi)=k\}$.
	Let us define a map $\phi:X\to Y$ as follows.
	For any $\pi\in X$, if $\pi$ is a signed permutation with even number of negative entries, define $\phi(\pi)=\pi$. Clearly, $\phi(\pi)\in\mathfrak{D}_n^{+}$. Since $\pi$ is smooth, $\pi_1+\pi_2<0$ if and only if $\pi_1<0$; similarly, $\pi_1+\pi_2>0$ if and only if $\pi_1>0$.
	Hence, $0\in\Des_{B}(\pi)$ if and only if $0\in \Des_D(\pi)$, which implies $\des_D(\phi(\pi))=\des_B(\pi)$.
	
	If $\pi=\pi_1\pi_2\cdots \pi_n\in X$ is a signed permutation with odd number of negative entries, define $\phi(\pi)=(-\pi_1)\pi_2\cdots \pi_n$. Then $\phi(\pi)\in\mathfrak{D}_n^{+}$. To verify $\des_B(\pi)=\des_D(\phi(\pi))$, we analyze the following cases:
	\begin{itemize}
		\item $\pi_1>\pi_2>0$: $0\not\in\Des_B(\pi)$ but $1\in\Des_B(\pi)$. In this case, we have $-\pi_1<\pi_2$ and $-\pi_1+\pi_2<0$, so $0\in\Des_D(\phi(\pi))$ but $1\not\in\Des_D(\phi(\pi))$;
		\item $\pi_2>\pi_1>0$: $\{0,1\}\not\in\Des_B(\pi)$. In this case, we have $-\pi_1<\pi_2$ and $ -\pi_1+\pi_2>0$, so $\{0,1\}\not\in\Des_D(\phi(\pi))$;
		\item $\pi_2<\pi_1<0$: $\{0,1\}\in\Des_B(\pi)$. In this case, we have $-\pi_1>\pi_2$ and $-\pi_1+\pi_2<0$, so $\{0,1\}\in\Des_D(\phi(\pi))$;
		\item $\pi_1<\pi_2<0$: $0\in\Des_B(\pi)$ but $1\not\in\Des_B(\pi)$. In this case, we have $-\pi_1>\pi_2$ and $-\pi_1+\pi_2>0$, so $0\not\in\Des_D(\phi(\pi))$ but $1\in\Des_D(\phi(\pi))$.
	\end{itemize}
	In all cases, the map $\phi$ preserves the number of descents. So, $\phi(w)\in Y$.
	To confirm $\phi$ is a bijection, we define the inverse map $\phi^{-1}:Y\to X$ by analysing the first two elements of $\pi\in Y$.
	If $\pi_1\cdot \pi_2>0$, set $\phi^{-1}(\pi)=\pi$. If $\pi_1\cdot \pi_2<0$, set $\phi^{-1}(\pi) = (-\pi_1)\pi_2\cdots \pi_n$. It is straightforward to verify that $\phi^{-1}$ is indeed the inverse of $\phi$.
	This completes the proof.
\end{proof}

We now turn to study of non-smooth signed permutations.

\begin{lemma}\label{lem-nonsmooth}
	For any $0\leq k\leq n$, the set of non-smooth signed permutations in $\mathfrak{B}_n^{+}$ with $k$ type $B$ descents is in bijection with the set of pairs $(u,\sigma)$, where $u\in [n]$ and $\sigma\in \mathfrak{B}_{n-1}^{-}$ with $\des_B(\sigma)=k$.
\end{lemma}
\begin{proof}
	We construct the map $\Phi$ from $\{\pi\in\mathfrak{B}_n^{+}\mid \des_B(\pi)=k, \pi \mbox{ is non-smooth} \}$ to $\{(u,\sigma)\in [n]\times\mathfrak{B}_{n-1}^{-}\mid \des_B(\sigma)=k\}$ via two steps as follows.
	Let $\pi=\pi_1\pi_2\cdots \pi_n\in \mathfrak{B}_n^{+}$.
	First, we define
	\begin{align*}
		\phi_1(\pi_1\pi_2\cdots \pi_n)=(|\pi_1|,\widetilde{\pi_2}\cdots \widetilde{\pi_n}),
	\end{align*}
	where $\widetilde{\pi_2}\cdots \widetilde{\pi_n}$ is the reduction of $\pi_2\pi_3\cdots \pi_n$ (see the definition of reduction preceding \eqref{eq-ab-for-red}).
	Second, define a transformation $\phi_n$ on a letter $x$ by
	\begin{align*}
		\phi_n(x)=x', \mbox{where $|x'|=n-|x|$ and $x'\cdot x<0$.}
	\end{align*}
	Now, define
	$$\Phi(\pi_1\pi_2\cdots \pi_n)=(|\pi_1|,\phi_n(\widetilde{\pi_2})\cdots \phi_n(\widetilde{\pi_n})).$$

	For example, take $\pi=4\bar{5} \bar{3}18\bar{6}72$ (here $\bar{a}=-a$). The standardization of $\pi_2\cdots \pi_8=\bar{5}\bar{3}18\bar{6}72$ is $\widetilde{\pi_2}\cdots \widetilde{\pi_8}=\bar{4}\bar{3}17\bar{5}62$.
	Applying $\phi_8$ to each letter gives $\phi_8(\widetilde{\pi_2})=4, \phi_8(\widetilde{\pi_3})=5, \phi_8(\widetilde{\pi_4})=\bar{7}, \phi_8(\widetilde{\pi_5})=\bar{1}, \phi_8(\widetilde{\pi_6})=3, \phi_8(\widetilde{\pi_7})=\bar{2}$ and $\phi_8(\widetilde{\pi_8})=\bar{6}$. Thus, $\Phi(4\bar{5}\bar{3}18\bar{6}72)=(4,45\bar{7}\bar{1}3\bar{2}\bar{6}).$
	
	By this construction, $\Phi(\pi)\in[n]\times \mathfrak{B}_{n-1}^-$.
	We also need to check $\des_B(\phi_n(\widetilde{\pi_2}) \cdots \phi_n(\widetilde{\pi_n}))=k$.
    Observe that reduction preserves relative order: for any maximal block $\pi_r \pi_{r+1} \cdots \pi_{r+\ell}$ inside $\pi_2 \cdots \pi_n$ consisting of entries with the same sign, the relative order within the block is unchanged in $\widetilde{\pi_2}\cdots \widetilde{\pi_n}$, and moreover, is preserved under the map $\phi_n$.
	Thus, descent changes could only occur between two consecutive entries with opposite signs. There are two possible cases: $\pi_1<0$, $\pi_2>0$ and $\pi_1>0$, $\pi_2<0$. In both situations, an argument identical to that used in the proof of Lemma~\ref{lem-2AB} shows that the number of type $B$ descents remain unchanged after applying $\Phi$.
Thus, $\Phi(\pi) \in \{(u,\sigma)\in [n]\times\mathfrak{B}_{n-1}^{-}\mid \des_B(\sigma)=k\}$.
Conversely, it is straightforward to define the inverse of $\Phi$ by reversing the above steps.
This completes the proof.
\end{proof}

		

Combining Lemma \ref{lem-half-smooth} and Lemma \ref{lem-nonsmooth}, we now prove Theorem \ref{thm-BDhalf} by analysing the smooth and non-smooth permutations in $\mathfrak{B}_n^+$, separately.
\begin{proof}[Proof of Theorem \ref{thm-BDhalf}]
	Any $w\in\mathfrak{B}_n^+$ is either smooth or non-smooth. Lemma \ref{lem-half-smooth} and Lemma \ref{lem-nonsmooth} imply
	\begin{align*}
		B_n^+(t)=D_n^+(t)+nB_{n-1}^-(t).
	\end{align*}
A similar argument to that of Lemma \ref{lem-half-smooth} establishes a bijection between the smooth signed permutations in $\mathfrak{B}_n^{-}$ with $k$ type $B$ descents and those of the permutations in $\mathfrak{D}_n^{-}$ with $k$ type~$D$ descents. Likewise, following  Lemma \ref{lem-nonsmooth}, the set of non-smooth signed permutations in $\mathfrak{B}_n^{-}$ with $k$ type $B$ descents is in bijection with the set of pairs $(u,\sigma)$, where $u\in [n]$ and $\sigma\in \mathfrak{B}_{n-1}^{+}$ with $\des_B(v)=k-1$. Consequently,
	\begin{align*}
		B_n^-(t)=D_n^-(t)+ntB_{n-1}^+(t),
	\end{align*}
	which completes the proof.
\end{proof}

Next, we present a new relation between Eulerian polynomials of type $A$ and type $B$ as a corollary of Theorem \ref{thm-BDhalf}.
To this end, we first establish the following lemma:
\begin{lemma}\label{lem-ABhalf}
	For any positive integer $n$, we have
	\begin{align}\label{equ-AB+-}
		2^{n}t A_{n}(t) = B_{n}^-(t) + t  B_{n}^+(t).
	\end{align}
\end{lemma}
\begin{proof}
	By Theorem \ref{thm-BDhalf}, we derive
	\begin{align*}
		B_n(t) - D_n(t) &= \left(B_n^+(t) + B_n^-(t) \right)- \left(D_n^+(t)  + D_n^-(t)\right)\\
		&= \left(B_n^+(t)- D_n^+(t)\right)+\left( B_n^-(t)-D_n^-(t)\right)\\
		&=n B_{n-1}^-(t) + n t  B_{n-1}^+(t).
	\end{align*}
	By equation \eqref{rela-ABD}, we have
	\begin{align*}
		B_n(t) - D_n(t) = n 2^{n-1} t  A_{n-1}(t).
	\end{align*}
	Equating the two expressions for $B_n(t) - D_n(t)$, we have
	\begin{align*}
		n B_{n-1}^-(t) + n t B_{n-1}^+(t)= n 2^{n-1} t A_{n-1}(t).
	\end{align*}
	Dividing both sides by $n$ and reindexing $n$ to $n+1$ lead to equation \eqref{equ-AB+-}, as required.
\end{proof}

Based on equation \eqref{equ-AB+-}, and the following relations for $B_n^+(t)$ and $B_n^-(t)$ given by Hyatt in \cite{H16},
\begin{align}
	B_n^+(t)&=\sum_{j=0}^{n-1}\binom{n}{j}B_{j}(t)(t-1)^{n-j-1}, \label{equ-B+-1}\\
	B_n^-(t)&=t^nB_n^+(t^{-1})=t\cdot \sum_{j=0}^{n-1}\binom{n}{j}B_{j}(t)(1-t)^{n-j-1},\label{equ-B+-2}	
\end{align}
we derive the following new result.

\begin{proposition}
	Let $n$ be a positive integer. If $n$ is even, then
	\begin{align*}
		2^{n-1}A_{n}(t)=\sum_{r=0}^{\frac{n-2}{2}}\binom{n}{2r+1}B_{2r+1}(t)(t-1)^{n-2r-2}.
	\end{align*}
    If $n$ is odd, we have
	\begin{align*}
		2^{n-1}A_{n}(t)=\sum_{r=0}^{\frac{n-1}{2}}\binom{n}{2r}B_{2r}(t)(t-1)^{n-2r-1}.
	\end{align*}
\end{proposition}
\begin{proof}
	From equation \eqref{equ-AB+-}, we have
\begin{align*}
	2^{n}tA_{n}(t) = B_{n}^-(t) + t B_{n}^+(t).
\end{align*}
Substituting equation \eqref{equ-B+-1} and equation \eqref{equ-B+-2} into the right-hand side gives
\begin{align*}
	2^{n}tA_{n}(t)=t \sum_{j=0}^{n-1}\binom{n}{j}B_{j}(t)(1-t)^{n-j-1}+t\sum_{j=0}^{n-1}\binom{n}{j}B_{j}(t)(t-1)^{n-j-1}.
\end{align*}
Divide both sides by $t$ to obtain
\begin{align*}
	2^{n}A_{n}(t)&=\sum_{j=0}^{n-1}\binom{n}{j}B_{j}(t)\left((1-t)^{n-j-1}+(t-1)^{n-j-1}\right)\\
		 &=\sum_{j=0}^{n-1}\binom{n}{j}B_{j}(t)(t-1)^{n-j-1}\left(1+(-1)^{n-j-1}\right).
\end{align*}
The desired result follows by considering the parity of $n$.
\end{proof}

The remainder of this section focuses on the polynomials $\mathbf{D}_{n,i}(t)$.
To present the relation between restricted Eulerian polynomials of type $D$ and Eulerian polynomials of type $A$ and type $B$, we first establish a recurrence for $\mathbf{D}_{n,1}(t)$.
\begin{lemma}\label{lemma-D1}
    For any $n\geq 1$, we have
	\begin{align*}
		\mathbf{D}_{n,1}(t)=tD_{n-1}^+(t) + D_{n-1}^-(t).
	 \end{align*}
\end{lemma}
\begin{proof}
	This result is obtained by analyzing the sign of $(n-1)$-th element in permutations belonging to $\mathfrak{D}_{n,1}$. If the $(n-1)$-th element is positive, a descent is introduced at position $n-1$, contributing $t D_{n-1}^+(t)$; If the $(n-1)$-th element is negative, no additional descent is created at position $n-1$, contributing $D_{n-1}^-(t)$. Summing these two contributions gives the desired identity.
\end{proof}

Combining Lemma \ref{lem-ABhalf}, Lemma \ref{lemma-D1}, and Theorem \ref{thm-BDhalf}, we are now ready to prove the following result.
\begin{theorem}\label{thm-DAB1}
	For any $n\geq 1$, we have
	\begin{align}\label{eq-DAB}
		\mathbf{D}_{n,1}(t)=t\cdot (2^{n-1}A_{n-1}(t)-(n-1)B_{n-2}(t)).
	\end{align}
\end{theorem}
\begin{proof}
	By Lemma \ref{lemma-D1} and Theorem \ref{thm-BDhalf}, we obtain
	\begin{align*}
		\mathbf{D}_{n,1}(t) &=t D_{n-1}^+(t) + D_{n-1}^-(t)\\
		                &=t \left(B_{n-1}^+(t)-(n-1) B_{n-2}^-(t) \right) +B_{n-1}^-(t) -(n-1) t B_{n-2}^+(t)\\
		                &=B_{n-1}^-(t) + t B_{n-1}^+(t) -(n-1) t \left(B_{n-2}^-(t) + B_{n-2}^+(t)\right)\\
		                &=B_{n-1}^-(t)  + t B_{n-1}^+(t) -(n-1) t B_{n-2}(t).
	\end{align*}
    Then by Lemma \ref{lem-ABhalf}, we have
	\begin{align*}
		\mathbf{D}_{n,1}(t) = 2^{n-1} t A_{n-1}(t) -(n-1) t B_{n-2}(t),
	\end{align*}
	as desired.
\end{proof}

In order to express $\mathbf{D}_{n,i}(t)$ in terms of $A_{n-1}(t)$ and $B_{n-2}(t)$ for all $1\leq i\leq n$, we establish the following recurrence for $\mathbf{D}_{n,i}(t)$.
\begin{proposition}\label{thm-D}
	For any $1\leq i\leq n$, we have
	\begin{align}\label{eq-D}
		\mathbf{D}_{n,i}(t)=\mathbf{D}_{n,i+1}(t) + (t-1)\mathbf{D}_{n-1,i}(t).
	\end{align}
\end{proposition}
\begin{proof}
	Define two subsets of $\mathfrak{D}_n$ as follows:
	$$X=\{w\in\mathfrak{D}_n\mid w_{n-1}=i, w_n=i+1\} \mbox{ and } \widetilde{X}=\{w\in\mathfrak{D}_n\mid w_{n-1}=i+1, w_n=i\}.$$
	We decompose the restricted Eulerian polynomials $\mathbf{D}_{n,i+1}(t)$ and $\mathbf{D}_{n,i}(t)$ by excluding or including these subsets, leading to:
	\begin{align}
		\mathbf{D}_{n,i+1}(t)&=\sum_{w\in D_{n,i+1}\backslash X}t^{\des_D(w)}+\sum_{w\in X}t^{\des_D(w)}.\label{eq-DD} \\
		\mathbf{D}_{n,i}(t)&=\sum_{w\in D_{n,i}\backslash \widetilde{X}}t^{\des_D(w)}+\sum_{w\in \widetilde{X}}t^{\des_D(w)}\label{eq-DD2}.
	\end{align}
	Notice that
	\begin{align*}
		\sum_{w\in D_{n,i+1}\backslash X}t^{\des_D(w)} = \sum_{w\in D_{n,i}\backslash \widetilde{X}}t^{\des_D(w)}.
	\end{align*}
	This holds because the only difference between $D_{n,i+1}$ and $D_{n,i}$ lies in the last two elements.
	Moreover,
	\begin{align*}
		t\cdot\sum_{w\in X}t^{\des_D(w)}= \sum_{w\in \widetilde{X}}t^{\des_D(w)}.
	\end{align*}
	This follows because swapping $w_{n-1} = i$ and $w_n = i+1$ creates an additional descent at position $n-1$.
    Subtracting equation \eqref{eq-DD} from equation \eqref{eq-DD2}, we obtain
	\begin{align*}
		\mathbf{D}_{n,i}(t)-\mathbf{D}_{n,i+1}(t)= (t-1)\cdot \sum_{w\in X}t^{\des_D(w)}=(t-1) \mathbf{D}_{n-1,i}(t).
	\end{align*}
	This completes the proof.
\end{proof}

Combining Theorem \ref{thm-DAB1} and Proposition \ref{thm-D}, we can express $\mathbf{D}_{n,i}(t)$ in terms of $A_{n-1}(t)$ and $B_{n-2}(t)$ for all $1\leq i\leq n$. By iterating equation \eqref{eq-D}, for $i\geq 2$, we have
\begin{align*}
	\mathbf{D}_{n,i}(t)=\mathbf{D}_{n,1}(t)+(i-1)\sum_{r=1}^{i-2}(1-t)^r\mathbf{D}_{n-r,1}(t)+(1-t)^{i-1}\mathbf{D}_{n-i+1,1}(t).
\end{align*}
Then substituting equation \eqref{eq-DAB} into the above equation, for $i\geq 2$, we obtain
{\small\begin{align*}
	\mathbf{D}_{n,i}(t)&=2^{n-1}tA_{n-1}(t)+(i-1)\sum_{r=1}^{i-2}2^{n-r-1}t(1-t)^rA_{n-r-1}(t)+2^{n-i}t(1-t)^{i-1}A_{n-i}(t)-\\
     &(n-1)tB_{n-2}(t) - (i-1)\sum_{r=1}^{i-2}(n-r-1)t(1-t)^rB_{n-r-2}(t) - (n-i)t(1-t)^{i-1}B_{n-i-1}(t).
\end{align*}}
When $i$ is a negative integer,  following an argument analogous to that of Lemma \ref{lem-B}, we can prove $\mathbf{D}_{n,i}(t)=t^n\mathbf{D}_{n,-i}(t^{-1})$. This allows us to express $\mathbf{D}_{n,i}(t)$ in terms of $A_{n-1}(t)$ and $B_{n-2}(t)$ for any $-n\leq i\leq -1$.
However, this formula is not concise as Theorem \ref{thm-DAB1}. We guess that a simpler relation exists between $\mathbf{D}_{n,i}(t)$ and  $\mathbf{A}_{n,i}(t)$ and  $\mathbf{B}_{n,i}(t)$ for any $i$. It is natural to ask the following problem.

\begin{problem}
	Find a direct relation between $\mathbf{D}_{n,i}(t)$, $\mathbf{A}_{n,i}(t)$, and $\mathbf{B}_{n,i}(t)$ for any $-n\leq i\leq n$.
\end{problem}

\noindent \textbf{Acknowledgments.}


We wish to thank Arthur L.B. Yang for  for the fruitful discussions and and Ethan Y.H. Li for his valuable suggestions on the revision of this paper. We would also like to thank the anonymous referees for their helpful comments.


\begin{thebibliography}{99}

\bibitem{AB25}  A. Abram and J. Bastidas, The $h^*$-polynomials of type $C$ hypersimplices, \emph{arXiv: 2504.038898}, 2025.

\bibitem{Brenti94} F. Brenti, $q$-Eulerian polynomials arising from Coxeter groups, \emph{European J. Combin.} 15 (1994), 417--441. 

\bibitem{BW08} F. Brenti and V. Welker, $f$-Vectors of barycentric subdivisions, \emph{Math. Z.} 259 (2008), 849--865. 

\bibitem{Chow22} C.-O. Chow, New proofs of interlacing of zeros of Eulerian polynomials, \emph{J. Math. Anal. Appl.} 510 (2022), 126019.

\bibitem{C2010} M.A. Conger, A refinement of the Eulerian numbers, and the joint distribution of $\pi(1)$ and $\Des(\pi)$ in $S_n$, \emph{Ars Combin.} 95 (2010), 445--472.

\bibitem{D23} H.K. Dey, Interlacing of zeroes of certain real-rooted polynomials, \emph{Arch. Math.} (Basel) 120 (2023), 457--466. 

\bibitem{DSS24} H.K. Dey, U. Shankar and S. Sivasubramanian, A descent-excedance correspondence in colored permutation groups, \emph{Enumer. Comb. Appl.} 5 (2025),  Article \#S2R24.

\bibitem{DPS09} K. Dilks, T.K. Petersen, and J.R. Stembridge, Affine descents and the Steinberg torus, \emph{Adv. Appl. Math.} 42 (2009), 423--444.

\bibitem{ERS98} R. Ehrenborg, M. Readdy and E. Steingr\'imsson, Mixed volumes and slices of the cube, \emph{J. Comb. Theory, Ser. A} 81 (1998), 121--126. 

\bibitem{F1910} G. Frobenius, \"Uber die Bernoullischen Zahlen und die Eulerschen polynome, Sitzungsberichte der \"oniglich Preußischen Akademie der Wissenschaften (1910), \emph{Zweiter Halbband}, pp. 809--847.

\bibitem{Guo24} W.M. Guo, Zeros distribution and interlacing property for certain polynomial sequences, \emph{Open Math.} 22 (2024), 20240085.

\bibitem{H16}  M. Hyatt, Recurrences for Eulerian Polynomials of Type $B$ and Type $D$, \emph{Ann. Comb.} 20 (2016), 869--881.

\bibitem{KN09} M. Kubitzke and E. Nevo, The Lefschetz property for barycentric subdivisions of shellable complexes, \emph{Trans, Amer. Math. Soc.} 361 (2009), 6151--6163.

\bibitem{LW07} L.L. Liu and Y. Wang, A unified approach to polynomial sequences with only real zeros, \emph{Adv. Appl. Math.} 38 (2007), 542--560. 

\bibitem{Petersen-book} T.K. Petersen, Eulerian Numbers. Springer, New York (2015).

\bibitem{MMYY22} S.-M. Ma, J. Ma, J. Yeh and Y.-N. Yeh, Eulerian pairs and Eulerian recurrence systems, \emph{Discrete Math.} 345 (2022), 112716.

\bibitem{NPT11} E. Nevoa, T.K. Petersenb and B.E. Tenner, The $\gamma$-vector of a barycentric subdivision, \emph{J. Comb. Theory, Ser. A} 118 (2011), 1364--1380.

\bibitem{P19} A. Postnikov, Permutohedra, associahedra, and beyond, \emph{Int. Math. Res. Not.} 6 (2019), 1026--1106.

\bibitem{S23} L. Santocanale, Bijective proofs for Eulerian numbers of types $B$ and $D$, \emph{Discrete Math. Theor. Comput. Sci.} 23 (2023), no. 2.  

\bibitem{SV15} C.D. Savage and M. Visontai, The $s$-Eulerian polynomials have only real roots, \emph{Trans. Amer. Math. Soc.}  367 (2015), 1441--1466.

\bibitem{Stembridge94} J.R. Stembridge. Some permutation representations of Weyl groups associated with the cohomology of toric
 varieties. \emph{Adv. Math.} 106 (1994), 244--301. 

\bibitem{VW13} M. Visontai and N. Williams, Stable multivariate W-Eulerian polynomials. \emph{J. Combin. Theory Ser. A} 120 (2013), 1929--1945.


\bibitem{WXX10} R.-H. Wang, Y. Xu and Z-Q. Xu, Eulerian numbers: A spline perspective, \emph{J. Math. Anal. Appl.} 370 (2010), 486--490.

\bibitem{YZ} A.L.B Yang and P.B Zhang,  Descent generating polynomials and the Hermite-Biehler theorem, \emph{J. Algebr. Comb.} 56 (2022), 117--152.

\bibitem{Zhang25} R. Zhang, The logarithm of the exponential generating function of Eulerian polynomials, \emph{Contrib. Discrete Math.}, to appear.

\bibitem{Zhu20} B.-X. Zhu, A generalized Eulerian triangle from staircase tableaux and tree-like tableaux, \emph{J. Comb. Theory, Ser. A}, 172 (2020), 105206.
\end{thebibliography}
\end{document}